\documentclass[a4paper,11pt,english]{amsart}

\usepackage{multirow}
\usepackage[utf8]{inputenc}
\usepackage[english]{babel}
\usepackage[T1]{fontenc}

\usepackage{graphicx}
\usepackage{hyperref}
\usepackage{amssymb}
\usepackage{amsmath}
\usepackage{verbatim,enumerate}
\usepackage{tikz}
\usepackage{accents}

\newtheorem{thm}{Theorem}[section]
\newtheorem{lem}[thm]{Lemma}
\newtheorem{prop}[thm]{Proposition}
\newtheorem{cor}[thm]{Corollary}

\theoremstyle{definition}

\newtheorem{rem}[thm]{Remark}
\newtheorem{defi}[thm]{Definition}

\newcommand{\C}{\mathbb{C}}
\newcommand{\CC}{\mathbb{C}}

\newcommand{\CP}{\mathbb{C}P}
\newcommand{\RP}{\mathbb{R}P}

\renewcommand{\epsilon}{\varepsilon}
\newcommand{\R}{\mathbb{R}}
\newcommand{\RR}{\mathbb{R}}
\newcommand{\Z}{\mathbb{Z}}

\newcommand{\ZZ}{\mathbb Z}

\usepackage{wasysym}

\newcommand{\ignore}[1]{\relax}

\usepackage{hyperref}

\newcommand{\Tor}{{\mbox{Tor}}}
\newcommand{\Vol}{{\mbox{Vol}}}

\begin{document}
\title{
  Non-existence of torically maximal hypersurfaces}
\author{Erwan Brugall\'e}
\author{Grigory Mikhalkin}
\author{Jean-Jacques Risler}
\author{Kristin Shaw}
 \address{Erwan Brugall\'e, \'Ecole Polytechnique,
Centre Mathématiques Laurent Schwartz, 91 128 Palaiseau Cedex, France}

\email{erwan.brugalle@math.cnrs.fr}
\renewcommand\rightmark{Non-existence of torically maximal hypersurfaces}
\renewcommand\leftmark{E. Brugall\'e, G. Mikhalkin, J.-J.~Risler, K.~Shaw}

\address{Grigory Mikhalkin, Section de math\'ematiques,
Universit\'e de Gen\`eve,
Villa Battelle,
1227 Carouge, Suisse}
\email{grigory.mikhalkin@unige.ch}

\address{Jean-Jacques Risler, Universit\'e Pierre et Marie Curie,
4 Place Jussieu, 75 005 Paris, France}
\email{jean-jacques.risler@imj-prg.fr}

\address{Kristin Shaw, Fields Institute for Research in Mathematical Sciences, 222 College Street
Toronto, M5T 3J1,
Canada}
\email{kshaw@fields.utoronto.ca}


 \thanks{Part of this work was done during the
research stay of E.B. and J.-J.R. at the Centre Interfacultaire Bernoulli (Lausanne)
in the framework of the program
``Tropical geometry in its complex and symplectic aspects''. The authors
are grateful to CIB for the support and excellent working conditions.
Research of G.M. is supported in part by the grant TROPGEO of the European Research Council,
by the grants 141329 and 159240 of the Swiss National Science Foundation,
and by the NCCR SwissMAP of the Swiss National Science Foundation.}

\date{\today} \maketitle

\begin{abstract}
Torically maximal curves (known also as simple Harnack curves)
are real algebraic curves in the projective plane such that their
logarithmic Gau{\ss} map is totally real.
In this paper we show that hyperplanes in projective spaces are the
only 
torically maximal hypersurfaces of higher dimensions.
\end{abstract}

\section{Introduction}

Torically maximal curves, also known as simple Harnack curves, were introduced  and studied in
\cite{Mik-00}.
Since then, they 
have appeared
in several areas of mathematics.
Finding their reasonable 
higher dimensional 
counterparts 
is an open and challenging problem
 (cf. e.g. \cite{AIM06}).
In this note we explore a direct generalisation of toric maximality
for projective hypersurfaces proposed in
\cite[Section 3.4]{Mik-amoebas}.
We show that when $n\ge 3$,
hyperplanes   in projective spaces are the only torically maximal
hypersurfaces
 in this sense.

\medskip

Let $X $ be an  algebraic hypersurface of $(\C^*)^n$ defined by the
equation $P(z_1,\ldots,z_n)=0$. We denote by $\Delta(X)$ the Newton
polytope of the polynomial $P(z_1,\ldots,z_n)$, and by
$\overline X$ the topological closure of $X$
in the toric variety $\Tor(X)$ defined by
$\Delta(X)$. Note that $X\subset (\C^*)^n$ determines $\Delta(X)$ only
up to a
translation
in $\Z^n$, however this does not play a role
in what follows.

\begin{defi}
We say that a hypersurface $X \subset (\C^*)^n$ is {\em torically non-singular}
if the polytope $\Delta(X)$ is $n$-dimensional
and the intersection of $\overline X$
with each torus orbit of $\Tor(X)$
is non-singular in this orbit. If $X$ is torically non-singular
then $\overline X$ is 
transverse to all
torus orbits of $\Tor(X)$.

We say that $X\subset (\C^*)^n$ is {\em torically projective} if
$\Tor(X)=\C P^n$.

\end{defi}

Following \cite{Ka-Gauss} we
define the \emph{logarithmic Gau{\ss} map} of
a non-singular hypersurface $X\subset (\C^*)^n$ as
$$\begin{array}{cccc}
 \gamma_X:&  X& \longrightarrow &\C P^{n-1}
\\ &(z_1,\ldots,z_n)&\longmapsto&[z_1\frac{\partial P}{\partial
    z_1}(z_1,\ldots,z_n):\ldots : z_n\frac{\partial P}{\partial
    z_n}(z_1,\ldots,z_n)]
\end{array}. $$
The map $\gamma_X$ is 
just
the usual Gau{\ss} map after the reparameterisation
of $X$ with the help of a local branch of the holomorphic logarithm restricted to $X$;
 clearly the map
  $\gamma_X$ does not depend on the chosen branch of the logarithm.
It is proved in 
\cite[Section 3.2]{Mik-00} that
when
$X$ is torically non-singular,
the map  $\gamma_X$ extends to an algebraic map $\overline
\gamma_X:\overline X\to
\C P^{n-1}$ such that 
\begin{equation}\label{eqn:degree}
\text{deg}(\overline{\gamma}_X) =  \Vol_n(\Delta(X)),
\end{equation}
where $\Vol_n$ denotes the
lattice volume of an $n$-dimensional polytope 
(i.e. $n!$  times the Euclidean volume).

Since   $\gamma_X$ is a
 map between $2$
  manifolds of the same
dimension, the fibre $\gamma_X^{-1}(y)$ is finite for almost all $y$ in $\CP^{n-1}$.
Our first result is that
$\gamma_X$  is actually finite 
in the case of
torically projective hypersurfaces, that is  $\gamma_X^{-1}(y)$ is
finite for \emph{any} $y\in \CP^{n-1}$. 

\begin{thm}\label{thm:finiteLog}
If $X \subset (\CC^*)^n$ is a torically non-singular  projective hypersurface, then the logarithmic Gau{\ss} map
$\overline \gamma_X: \overline X \to \CP^{n-1}$ is finite. 
\end{thm}

We then investigate the existence of \textit{torically 
  maximal hypersurfaces}.
Given a real algebraic subvariety   $X$ of  a
complex toric variety, we denote
 by $\R X$
the real part of $X$.
We say that a
real algebraic map $f:X\to Y$ between
$2$ real algebraic varieties
is
\emph{almost totally real}
if ${f}^{-1} (x) \subset \R  X$
for any $x\in  \R Y\setminus  S$ where $S$ is some
subspace of $f(X)\cap\R Y$ of positive codimension.
If $S$ is empty, then the map $f$ is said to be \emph{totally real}.

\begin{defi} \label{def:H-surface}
A torically non-singular real algebraic
hypersurface $X$ of $(\C^*)^n$ is
said to be \emph{almost torically maximal}
if the map
$\gamma_X$ is
almost
totally real.

A torically non-singular real algebraic
hypersurface $X$ of $(\C^*)^n$ is
said to be \emph{torically maximal}
if $\overline X$ is non-singular and the map
$\overline \gamma_X$ is  totally real.
\end{defi}

\begin{rem}
Note that 
the above definition of
(almost)
torically maximal is axiomatizing \cite[Proposition 26]{Mik-amoebas} rather than
making use of \cite[Definition 10]{Mik-amoebas}.
In particular, we do not require
$X$ or $\overline X$
 to be maximal in the sense the Smith-Thom
inequality (see for example \cite{BenRis} for the Smith-Thom inequality).
\end{rem}

Any
almost
torically maximal hypersurface
is
torically maximal if $n\le 2$.
When $n=1$, the variety
$X$ is
torically maximal if and only if all roots of
$P(z_1)$ in $\C^*$ are simple and real.
When $n=2$, the real curve $X$ is
torically maximal if and only if $\R
X$
is a simple Harnack curve, see \cite[Lemma 2.2 and Theorem 3.5]{PasRis11}.
It is proved in \cite{Mik-00} that the topological type of the pair
$((\R^*)^2, \R X)$ is uniquely determined by $\Delta(X)$
when $X$ is a simple Harnack curve (see also
\cite{Bru14b} for an alternative proof).

The  logarithmic
Gau{\ss} map of a hyperplane in $\CP^n$ has  degree 1, therefore any
 real hyperplane  is
almost
torically maximal,
and hence torically maximal by
Theorem 
\ref{thm:finiteLog}.
The next theorem 
asserts that this is the only possible example of
almost
torically projective 
maximal hypersurfaces
as soon as  $n\ge 3$.
\begin{thm} \label{thm:noProjHarnack}
  Let $n\ge 3$ and $X\subset (\C^*)^n$ be an
almost
  torically maximal 
  projective
  hypersurface.
  Then
  $\overline X$ is a hyperplane.
\end{thm}

 In the case of
 torically maximal hypersurfaces, the previous  theorem
can be extended to any Newton polytope.
\begin{thm} \label{thm:noGenHarnack}
  Let $n\ge 3$ and $X\subset (\C^*)^n$ be a
  torically 
maximal hypersurface.
Then $\Tor(X)= \C P^n$ and  $\overline X$ is a hyperplane.
\end{thm}

\begin{rem}
Note that Theorem \ref{thm:noProjHarnack} can be deduced as a
corollary of
Theorem  \ref{thm:finiteLog} and Theorem \ref{thm:noGenHarnack}.
 Nevertheless, its direct proof is quite simple, so
we prove it independently of Theorem \ref{thm:noGenHarnack}.
\end{rem}

Let us make some comments about Theorems \ref{thm:noProjHarnack} and
\ref{thm:noGenHarnack} and
further generalisations of simple Harnack curves.
First,
%
we
do not know
whether   there  exist almost 
 torically
 maximal hypersurfaces $X\subset (\C^*)^n$
 which are not torically maximal.
 However,  Section \ref{sec:example} provides an example of a singular hypersurface for which the logarithmic 
 Gau{\ss} map is almost totally real but not totally real. 
Therefore the assumption of smoothness of $\overline X$  (which is a part of the definition of toric maximality) is essential
 in Theorem \ref{thm:noGenHarnack}.

 Next, Theorems \ref{thm:noProjHarnack} and \ref{thm:noGenHarnack} may be
 a hint 
 that the direct generalisation of toric maximality
 proposed in
 \cite[Section 3.4]{Mik-amoebas} in dimension at least $3$
 can be weakened.
 For example, relaxing the smoothness assumption on $X\subset  (\C^*)^n$ in
 Definition \ref{def:H-surface} may
 produce  meaningful objects (see  \cite{Lang15} for the case of
generalised simple
   Harnack curves). 
   Additionally, it is worthwhile to consider
real subvarieties of higher codimension.  
There is a
natural
generalisation of the logarithmic Gau{\ss}
map where the target is now a  Grassmannian,
and also a generalisation of (almost) torically maximal real algebraic varieties of any codimension.
Products of torically maximal hypersurfaces 
give examples 
of  torically maximal subvarieties of codimension $> 1$.
So far  we do not know of other examples.  



\medskip



\medskip
{\bf Acknowledgment: }We are grateful to Benoît Bertrand, Christian
Haase, Ilia
  Itenberg,   Michael Joswig, Mario Kummer, and Lucia L\'opez de Medrano
 for helpful discussions.

\section{Properties of the logarithmic Gau{\ss} map}\label{sec:toricallymaxcoveringmap}

If $X \subset (\CC^*)^{n}$ is a
torically non-singular
hypersurface, then  $\overline{X} \cap Y$ is by definition also a
torically 
non-singular hypersurface for any  torus orbit
$Y$
of $\Tor(X)$. In particular the
logarithmic Gau\ss \  map
$\gamma_{\overline{X}\cap Y}$
 is
 well defined.
 The following
lemma is straightforward. 
\begin{lem}\label{lem:restrictionLogGauss}
Let $X \subset (\CC^*)^{n}$ be a torically non-singular hypersurface.
Then for any torus orbit
$Y$
of $\Tor(X)$,  the logarithmic Gau\ss \  map  of  $\overline{X} \cap
Y$ 
coincides with the restriction of ${\overline{\gamma}}_X$ to
$\overline{X} \cap Y$.
Furthermore, if the face  of $\Delta(X)$ corresponding to $Y$
is parallel to a
linear space 
$L \subset \R^n$, then the image of
the restriction of 
${\overline{\gamma}}_X$ to $\overline{X} \cap Y$ 
 lies in the projectivisation of
$L \otimes\C$ in $\CP^{n-1}$.  
\end{lem}

\begin{proof}[Proof of Theorem \ref{thm:finiteLog}]
   Lemma \ref{lem:restrictionLogGauss} implies that
 for any $x\in \CP^{n-1}$, the fibre
$\overline{\gamma}_X^{-1}(x)$ is disjoint from at least $1$ toric
divisor of $\CP^n$, which is a hyperplane.
Since any positive-dimensional subvariety of $\CP^n$
intersects any hyperplane, 
all fibres
$\overline{\gamma}_X^{-1}(x)$ have to be 
a 
finite collection of points. 
\end{proof}

\begin{rem}
  When $X$ is
torically non-singular, but not necessarily projective, the above argument can be used to show that any curve contained in the fibre $\overline{\gamma}_X^{-1}(x)$ 
must be contained in the closure of
 a subtorus translate of
 $(\C^*)^n$.
\end{rem}

Theorem \ref{thm:finiteLog} immediately implies the following.
 \begin{cor}\label{prop:maxfinite-totreal}
If $X \subset (\C^*)^n$ is an almost torically
maximal  projective hypersurface,
then $X$ is  torically maximal.
 \end{cor}

The  following theorem about totally real morphisms is used to restrict the topology of $\R \overline{X}$. Note that in \cite{Kummer} a totally real morphism is called real fibered.

\begin{thm}\cite[Theorem 2.19]{Kummer}\label{thm:covering}
  Let $X$ and $Y$ be   
  non-singular real algebraic varieties of the
  same dimension, and
  let $\phi: X \to Y$ be a  totally real morphism.
  Then $d_x \phi: T_x\R X \to T_{\phi(x)}\R Y$ is an isomorphism for all $x \in \R X$. 
\end{thm}

We outline the proof of the above theorem for completeness, referring the reader to \cite{Kummer} for details.
Since it is a local statement, we may assume that both $X$ and $Y$ are
real open
neighbourhoods of 0 in  $\C^n$.
Firstly, notice that the statement is true when $X$ and $Y$ are
$1$-dimensional:
if a real map    $\phi: (\CC, 0) \to (\CC, 0)$ is ramified at $0$,
then $\phi$ is  locally given by $z \mapsto z^d$ for
$d\ge 2$
which is clearly not totally real. 

For $n>1$, if $d_x\phi$ is not injective for some $x\in\R X$, then
choose a  real line $L\subset\C^n$
such that $\phi(x)\in L$ and $T_{\phi(x)}\R L\cap d_x \phi(T_x \R X ) = \{0\}$.
Consider the
real algebraic curve $C = \phi^{-1}(L) \subset X$, and its normalisation $\pi: \tilde{C} \to C$. The
composition $\phi \circ \pi : \tilde{C} \to L$ is also a totally real
map.
Since the
theorem is true for maps between curves, this map is unramified over
the real locus.
However, for any  point $\tilde{x} \in \R \tilde{C}$ such that $\pi (\tilde{x} ) = x \in \R C \subset \R X$, the differential satisfies $d_{\tilde x} (\phi \circ \pi)  = d_x \phi \circ d_{\tilde{x}} \pi$. Therefore the image of $d_{\tilde x} (\phi \circ \pi)$ is zero by the assumption that $T_{\phi(x)}\R L \cap d_x \phi(T_x \R X ) = \{0\}$. This gives a contradiction and the theorem follows.

\medskip
It follows from Theorem \ref{thm:covering} that
the logarithmic Gau{\ss} map induces a covering map
$\R\overline{X}\to \RP^{n-1}$ if $X\subset (\C^*)^n$ is a
torically
maximal  hypersurface.
For $n>1$, 
 there are only $2$ connected coverings of $\R P^n$, namely $\R P^{n}
\to \R P^{n}$ of degree $1$ and  $S^{n} \to \R P^{n}$ of degree $2$.
Hence 
the degree of 
the covering map $\R\overline{X}\to \RP^{n-1}$ 
is determined by the topology of $\R\overline{X}$ when $n\ge 3$, and  
Formula (\ref{eqn:degree})
implies
the following.

\begin{cor}\label{cor:SnRPn}
  Let $X \subset (\CC^*)^{n}$ be a
  torically  maximal
hypersurface
with $n\ge 3$. Then $\R \overline{X}$ is a disjoint union of $k$
connected components homeomorphic to 
$S^{n-1}$, and $l$
connected components homeomorphic to 
$\R P^{n-1}$. Furthermore, the integers $k$ and $l$ satisfy
$$\deg(\gamma) =  \Vol_{n}(\Delta(X)) = 2k + l.$$ 
\end{cor}

\section{Torically maximal hypersurfaces}\label{sec:surfaces}

Let $X \subset (\CC^*)^3$ be a
torically maximal surface.
By Lemma \ref{lem:restrictionLogGauss},
for each
2-dimensional torus orbit $Y$ of $\Tor(X)$, 
the curve $Z = \overline{X} \cap Y$
is a
simple Harnack
curve.
By \cite{Mik-00}, the intersections of $\overline Z$ with the toric
boundary divisors of 
$\Tor(Z)$
are real and contained in a single component of $\R\overline Z$.
Call this connected component the
\textit{outer
circle} of the curve $\overline Z$  and  denote it by
$O( Z)$.

\begin{lem}\label{lem:connCompCurve}
  Let $X \subset (\CC^*)^3 $ be a
torically maximal surface.
Then there exists a connected component of $ \R\overline{X} $
containing
all outer circles
of $\overline{X}$.   
\end{lem}
We call this connected component of $ \R\overline{X} $ the \emph{outer
  component} of $X$.
\begin{proof}
  Each 
outer circle 
is an embedded
circle
contained in  some connected component of 
$\R  \overline{X}$. Facets $F$ and $F'$ of $\Delta(X)$ intersect in an
  edge $E$ if and only if the corresponding  outer circles $O(Z)$ and
  $O(Z')$ intersect transversally in exactly $\text{Length}(E)$
  points, where
$\text{Length}(E)$ is the lattice  length of $E$, i.e.
$$\text{Length}(E) := |E \cap \ZZ^{3}| - 1.$$
In particular,
  $O(Z)$ and $O(Z')$ are contained in the same
  connected component of $\R \overline{X}$. The facets of  $\Delta(X)$
  are connected via the edges, therefore there exists a single
   connected component $\R \overline{X}$ containing  the outer circles of all
  boundary curves of $\overline{X}$. 
\end{proof}

\begin{prop}\label{prop:CisRP2}
  Let $X \subset (\CC^*)^3$ be
  a  torically maximal surface.
Then  the
outer component of $X$
is homeomorphic to $\R P^2$, 
and
$\Delta(X)$
is a tetrahedron with all edges of lattice length $1$. 
\end{prop}

\begin{proof}
  Let us denote by $\mathcal{C}$ the outer component of $X$.
  By Corollary \ref{cor:SnRPn},
 it is homeomorphic to either $S^2$ or $\RR P^2$. 
 Suppose
$\mathcal{C}$
 is homeomorphic to $S^2$.
Since any $2$ closed curves in $S^2$ intersecting transversally
do so  in an even number of points, we deduce that 
each edge of $\Delta(X)$ has an even lattice length.
A facet $F$
of $\Delta(X)$ has at least $3$ edges, therefore  
the
lattice
perimeter of every facet $F$ satisfies
$$\sum_{E \in \mathcal{E}(F)} \text{Length}(E) \geq 6,$$
where
$\mathcal{E}(F)$ denotes the set of edges of $F$.

On the other hand, by \cite{Mik-00} we have
\begin{equation}\label{eqn:degreePerimeter}
  \deg({\overline{\gamma}}_X|_{O(Z)}) =  \sum_{E \in \mathcal{E}(\Delta(Z))} \text{Length}(E)-2, 
\end{equation} 
for any outer circle $O(Z)$ of $\overline X$.
Since the  restriction of the logarithmic Gau\ss \  map
${\overline{\gamma}}_X$ to
$\mathcal C$ has degree 2,
 Equation  $(\ref{eqn:degreePerimeter})$ gives
$$ \sum_{E \in \mathcal{E}(\Delta(Z))} \text{Length}(E)-2  \le 2. $$
Therefore,
$\sum_{E \in \mathcal{E}(F)} \text{Length}(E) \le 4$ for any facet $F$ of $\Delta(X)$,
which yields a contradiction to the lower bound of the lattice perimeter of $F$ given above.

So  $\mathcal{C}$ is homeomorphic to $\RR P^2$ and the restriction of
the logarithmic Gau\ss \  map ${\overline{\gamma}}_X|_{\mathcal C}: \mathcal{C} \to \R P^2$ is
$1$-$1$.
Equation $(\ref{eqn:degreePerimeter})$
gives
$$\sum_{E \in \mathcal{E}(\Delta(Z))}
\text{Length}(E)   -2= 1,$$
which implies that each facet $F$ of $\Delta(X)$ is a
lattice triangle, and that $\text{Length}(E) = 1$ for all edges $E$ of
$\Delta(X)$.  In particular, each outer circle $O(Z)$ intersects
some other
outer circle $O(Z')$
%
 transversally
 in a single point. Hence each outer circle
 realises the non-zero class in $H_1(\mathcal C;\Z/2\Z)$.
But then any $2$ outer circles intersect, that is to say
each pair of faces of $\Delta(X)$ must share an edge.
This implies that $\Delta(X)$ is a lattice tetrahedron,
and the proposition is proved. 
\end{proof}

\begin{proof}[Proof of Theorem \ref{thm:noProjHarnack}]
  By 
  Corollary  \ref{prop:maxfinite-totreal} 
  the hypersurface  $X$ is
  torically maximal.
So the case $n= 3$ follows immediately from 
Proposition \ref{prop:CisRP2}. If $X \subset (\CC^*)^n$ is
torically maximal and projective for $n> 3$ 
then by intersecting $\overline {X}$ with a $3$-dimensional
torus orbit of $\Tor(X)$,
we would obtain 
a torically maximal  projective surface, 
which by above is a plane.
Therefore $\overline X$ must be a hyperplane. 
\end{proof}

Recall that given  a 
lattice polytope  $F$ of dimension $k$ in $ \R^{n}$, its
\emph{lattice volume} is defined as
$$\mathrm{Vol}_{k}(F) = \frac{\mathrm{Vol}^E_k(F)}{\mathrm{Vol}^E_k(\Pi_F)}, $$
where $ \mathrm{Vol}^E_k$ denotes
any
Euclidean volume in the affine span $V_F$ of $F$, and $\Pi_F$ is any
lattice simplex whose vertices form an affine basis of $V_F\cap \Z^n$.
We say that $F$ is unimodular if $\mathrm{Vol}_k(F) = 1$.

An $n$-dimensional lattice polytope $\Delta \subset \R^{n}$ is said to
be smooth in dimension $1$ if for every $1$-dimensional face $E$ of
$\Delta$,
there exist $n-1$ outward primitive integer normal vectors to the facets
adjacent to $E$ that can be completed
to a basis of $\Z^n$. 
If $\Delta$ is smooth in dimension $1$, then the corresponding toric
variety has singularities only at
$0$-dimensional torus orbits.
If $X \subset (\C)^n$ is a torically maximal hypersurface, then 
its Newton polytope $\Delta(X)$ is smooth in dimension $1$ since
$\overline X$ is non-singular.

\begin{lem}\label{lem:simplexUni}
Let $n\ge 3$, and let $\Delta \subset \R^{n}$ be an $n$-dimensional  lattice simplex
smooth in dimension $1$ such that all its facets are unimodular. 
Then 
$\Delta$ is unimodular. 
\end{lem}
\begin{proof}
  Denote by $(e_1,\ldots,e_n)$  the canonical basis of $\R^n$,
and   choose a facet $F$ of $\Delta$. 
 Since $F$ is unimodular, it can be  assumed, up to an integer affine
 transformation of $\R^n$, that 
 the vertices of $F$ are
$ 0, e_1, \dots , e_{n-1}$.
 There is $1$ additional vertex $a$ of $\Delta$ with
 $$a = (a_1, \dots, a_{n-1}, v) \in \ZZ^{n}.$$
 Note that  $\text{Vol}_n(\Delta) = |\mathrm{det}(e_1, \dots , e_{n-1}, a)| =  v$.

 By assumption,  the
 facet of $\Delta$ which is the convex hull of
all vertices of $\Delta$ except $e_i$
 is also unimodular, so there is a vector $c = (c_1, \dots , c_{n}) \in \ZZ^{n}$ such that 
$$\mathrm{det}(c, e_1, \dots, \hat{e_i} , \dots, e_{n-1}, a) = \pm (vc_i - a_ic_n) = \pm 1.$$
 Therefore,
 the primitive integer normal vectors to this  facet 
 are
$\pm (a_ie_{n} - ve_i)$.

The condition that $\Delta$ is smooth in dimension $1$ implies that at
each edge $E$ of $\Delta$, 
the
primitive integer outward
normal vectors of the facets of $\Delta$ adjacent to $E$
form
a subset of a  basis of $\Z^{n}$.  
Applying this condition at the edge $[0, e_1]$ of $\Delta$,
we deduce that
there must exist $c = (c_1, \dots , c_{n}) \in \ZZ^{n}$, such that
$$\text{det}(c, a_2e_{n} - ve_2, \dots, a_{n-1}e_{n} - ve_{n-1}, e_{n}) = \pm c_1\cdot v^{n-2} = \pm 1.$$ 
Therefore,  $\text{Vol}_n(\Delta) = v = 1$ and $\Delta$ is unimodular
as stated.
\end{proof}

\begin{rem}
  In dimension $3$, there are tetrahedra, with unimodular faces which are not unimodular. For example, the convex hull of 
$$(0, 0, 0), (1, 0, 0), (0, 1, 0), (1, p, q), $$
for every pair of $p, q$ with $\mathrm{gcd}(p, q) = 1$, has unimodular facets but the volume of the tetrahedron is $q$. 
So this polytope is not unimodular for $q>1$. Notice that it fails to be smooth in dimension $1$ along the edge joining $(1, 0, 0)$ and  $(0, 1, 0)$. 
\end{rem}

\begin{proof}[Proof of Theorem \ref{thm:noGenHarnack}]
The theorem is proved by induction on $n$, starting with $n= 3$ as the
base case.
Recall that $\Delta(X)$ is smooth in dimension $1$ if $X \subset (\C)^n$ is a
torically maximal hypersurface.

Let $X\subset (\C^*)^3$ be a
torically maximal surface.
By Corollary \ref{cor:SnRPn}, the real part $\R \overline{X}$ is a disjoint union of $k$
connected components homeomorphic to 
$S^{n-1}$ and $l$
connected components homeomorphic to 
$\R P^{n-1}$, 
such that
$$  \Vol_{n}(\Delta(X)) = 2k + l.$$ 
By Proposition \ref{prop:CisRP2}, the
outer component of $X$ is homeomorphic to
$\R P^2$, and
$\Delta(X)$
is a tetrahedron with all edge lengths equal to
$1$.
In particular we have $\sum_{E \in \mathcal{E}(\Delta)} \text{Length}(E) = 6$.

Let us denote by
$\beta_{\ast}(M; \Z/2\Z)$  the sum of all $\Z/2\Z$ Betti numbers
of a manifold $M$.
The Smith-Thom inequality states that (cf e.g. \cite{BenRis})
\begin{eqnarray}\label{eqn:smiththom}
\beta_*(\R \overline{X}; \Z / 2\Z) \leq \beta_*(\overline{X}; \Z/2\Z). 
\end{eqnarray}
The total sum of Betti numbers for $\RR P^2 $ and $S^2$ are  
$$ \beta_{\ast}(\R P^2; \Z/ 2\Z) = 3 \qquad \text{and} \qquad \beta_{\ast}(S^2; \Z/2\Z) = 2, $$
so that 
\begin{eqnarray}\label{eqn:betaR}
\beta_*(\R \overline{X}; \Z /2 \Z) = 3l + 2k  = 2l +  \text{Vol}_3(\Delta).
\end{eqnarray}

Moreover,  Khovanskii's formula \cite{Kho78} for the Euler
characteristic of the complex hypersurface $\overline{X}$ gives 
\begin{eqnarray}\label{eqn:Euler}
\end{eqnarray}
$$ \beta_*(\overline{X}; \Z/2\Z) = \text{Vol}_3(\Delta(X)) -  \sum_{F \in \mathcal{F}(\Delta(X))} \text{Area}(F)  +  \sum_{E \in \mathcal{E}(\Delta(X))} \text{Length}(E),$$
where $\mathcal{F}(\Delta(X))$ denotes the set of facets of $\Delta(X)$.
Combining Equations
(\ref{eqn:betaR}), (\ref{eqn:Euler}),
(\ref{eqn:smiththom})  
yields 
\begin{align*} 
2l +  \text{Vol}_3(\Delta)  \leq \text{Vol}_3(\Delta) -  \sum_{F \in \mathcal{F}(\Delta)} \text{Area}(F)  +  \sum_{E \in \mathcal{E}(\Delta)} \text{Length}(E), 
\end{align*}
which further implies   
that 
\begin{align*}\label{eqn:inequality}
 \sum_{F \in \mathcal{F}(\Delta)} \text{Area}(F)   \leq  4.
\end{align*} 

Therefore, 
 each facet of $\Delta$ is unimodular.  
Since $\Delta$ is also non-singular in dimension $1$,   Lemma
\ref{lem:simplexUni} implies that  $\Delta $ is itself
unimodular.
Hence $\Tor(X)=\CC P^3$ and $\overline{X}$ is a hyperplane. 
 
 \medskip
 Now proceed by induction for $n > 3$. Suppose $X \subset (\CC^*)^{n}$
 is a
 torically maximal hypersurface.
By
induction,
each facet $F$ of
$\Delta(X)$ is unimodular,   the
corresponding toric divisor $T_F$
is
$\CC P^{n}$, and
$\overline{X} \cap T_F$
is a hyperplane. 
In particular, the intersection 
$\RR \overline{X} \cap T_F$
is connected for all facets $F$.  
  
Therefore, similarly to Lemma \ref{lem:connCompCurve}, there is a single
connected component $\mathcal{C}$ of $\RR \overline{X}$ which contains
all
intersections  $\RR \overline{X} \cap T_F$ when $F$ runs over all
faces  of $\Delta(X)$.
Let $F$ be a facet of $\Delta(X)$, and let $A$ be
a $2$-dimensional face  of $\Delta(X)$ intersecting  $F$ along  an edge
$E$. Hence $ \RR \overline{X} \cap T_F$ and
$ \RR\overline{X} \cap T_A$ intersect transversally, and their
intersection is $\RR \overline{X} \cap  T_E$ which is a  single
point by the unimodularity of $F$.  
Hence  
the class realised by
$ \RR \overline{X} \cap  T_F$ in
$H_{n-2}(\mathcal{C}; \ZZ / 2\ZZ)$
is non-trivial. In particular $H_{n-2}(\mathcal{C}; \ZZ / 2\ZZ)\ne 0$, and
 $\mathcal{C}$ is homeomorphic to
$\RR P^{n-1}$.
Furthermore, for any $2$ facets $F$ and $F'$ of $\Delta(X)$, the
intersection of $ \RR \overline{X} \cap T_F$ and
$ \RR\overline{X} \cap T_{F'}$ realises the non-zero homology class in
$H_{n-3}(\mathcal{C}; \ZZ / 2\ZZ) =\Z / 2 \Z.$
In particular, the intersection  is non empty and of dimension $n-3$.
On the other hand, the intersection
of  
$\RR \overline{X} \cap T_F$ and 
$\RR \overline{X} \cap T_{F'}$ 
is of codimension $2$ if and only if $F$ and $F'$ intersect in a face
of $\Delta(X) $ of codimension $2$. This implies that  every pair of
facets of $\Delta(X)$ must meet in a codimension $2$ face. Therefore, the 
polytope 
$\Delta(X)$ has at most  $n+1$ facets, all of which are ${n-1}$
dimensional unimodular lattice simplicies. Since $\Delta(X)$ is also
smooth in dimension $1$ by assumption, applying Lemma
\ref{lem:simplexUni} completes the proof.   
\end{proof}

\section{A singular torically maximal surface}\label{sec:example}

We end the paper with an example showing that
the hypothesis that any 
 singularities of $\Tor(X)$  are contained in the
 $0$-dimensional torus orbits  is 
 essential in Theorem 
 \ref{thm:noGenHarnack}.
 Let   $\Delta \subset \RR^3$  be the simplex with  vertices
$$(0, 0, 0), (1, 0, 0), (0, 1, 0), (0, 0, 2),$$
and let  $X\subset  (\C^*)^3$ be a non-singular real algebraic
surface with $\Delta(X)=\Delta$. Up to a real toric change of
coordinates, the surface  $X$ has equation
$$ az_3^2 + z_3 + z_2 +z_1 + 1=0 $$
with $a\in  \R^\times$.

The variety  $\Tor(X)$ is singular along the orbit $Y$ corresponding to the
edge $e=[(1, 0, 0),  (0, 1, 0)]$.
%
%
%
Namely, the surface
$\overline X$ has an
ordinary double point at $p=\overline X\cap Y$. The
blow-up of $\Tor(X)$ along $Y$ is a non-singular toric variety $Z$, and
the
proper transform $\widetilde{X}$ of $\overline X$
is non-singular.
Note that $\widetilde X$ is simply the blow-up of $\overline X$ at the
point $p$.  We denote by $C$ the corresponding $(-2)$-curve in
$\widetilde X$. The logarithmic Gau{\ss} map
$\gamma_X:X\to\CP^2$ extends to a map
$\widetilde \gamma_X:\widetilde{X}\to\CP^2$ that contracts
the exceptional curve $C$ to a point.
In particular, the 
map $\overline \gamma_X:\overline X\to\CP^2$ is the composition of the
blow-down map with the map 
${\widetilde {\gamma}}_X$.

\begin{prop}
The map ${\overline{\gamma}}_X$ is totally real for $a\in(0,\frac{1}{4})$.
\end{prop}
Note that even if $a\in(0,\frac{1}{4})$, the map $\widetilde \gamma_{X}$ is
almost
totally real, but not totally real since it
contracts the curve $C$ to a point.
\begin{proof}
 One has
 $$\gamma_X(z_1,z_2,z_3)=[z_1:z_2:2az_3^2+z_3].$$
Given $(\gamma_1:\gamma_2:\gamma_3 )\in \R^3\setminus\{(0,0,0)\}$, determining the real
points in the fibre
$\gamma_X^{-1}([\gamma_1:\gamma_2:\gamma_3 ])$ reduces to solve the
 system
$$(S)\quad \left\{
\begin{array}{rrr}
  az_3^2 + z_3 + z_2 +z_1 + 1 &=&0
  \\ z_1 &=& s\gamma_1
  \\ z_2&=&s\gamma_2
  \\ 2az_3^2+z_3&=&s\gamma_3
   \end{array}
\right.$$
 in the variables $z_1,z_2,z_3\in \R$ and $s\in \R^*$.

 Since the triangle with vertices $(0,0,0), (1,0,0), (0,1,0)$ is
 unimodular,  it follows from Lemma \ref{lem:restrictionLogGauss} that
 ${\overline{\gamma}}_X^{-1}([\gamma_1:\gamma_2:0 ])$ contains at least $1$ real
 point. Since it is a degree 2 map, the whole fibre must be contained in
 $\R\overline X$.
 Similarly, we have
 ${\overline{\gamma}}_X^{-1}([\gamma_1:\gamma_2:\gamma_3 ])\subset
 \R\overline X$
 if $2\gamma_1+2\gamma_2+\gamma_3=0$.
 
Assume now that $\gamma_3=1$ and  $2\gamma_1+2\gamma_2+1\ne 0$.
Then the system $(S)$ reduces to the system
$$\left\{
\begin{array}{rrr}
  -az_3^2 + s(\gamma_1+\gamma_2+1) + 1 &=&0
  \\ 2az_3^2+z_3&=&s
   \end{array}
\right. .$$
Substituting $s=2az_3^2+z_3$ in the first equation we obtain
$$a(2\gamma_1+2\gamma_2+1)z_3^2 +  (\gamma_1+\gamma_2+1)z_3 +1=0.$$
This is a degree 2 equation in the variable $z_3$ whose discriminant
is
$$ (\gamma_1+\gamma_2+1)^2 -4a(2\gamma_1+2\gamma_2+1) =
(\gamma_1+\gamma_2+1)^2 -8a(\gamma_1+\gamma_2+1) +4a. $$
The polynomial $P(x)=x^2 -8ax +4a$ has discriminant
$$16a(4a-1), $$
and so is negative if $a\in(0,\frac{1}{4})$. In this case
$P (\gamma_1+\gamma_2+1)$ is positive, and
${\overline{\gamma}}_X^{-1}([\gamma_1:\gamma_2:1 ])$ is composed of
$2$ points in $\R\overline X$.
\end{proof}

\bibliographystyle {alpha}
\bibliography {Biblio.bib}

\def\cprime{$'$}
\begin{thebibliography}{AIM06}

\bibitem[AIM06]{AIM06}
Extreme forms of real algebraic varieties.
\newblock Am. Inst. of Math. workshop, April 6th-9th 2006.

\bibitem[BR90]{BenRis}
R.~Benedetti and J.~J. Risler.
\newblock {\em Real algebraic and semi-algebraic sets}.
\newblock Actualit\'es Math\'ematiques. [Current Mathematical Topics]. Hermann,
  Paris, 1990.

\bibitem[Bru15]{Bru14b}
E.~Brugall{\'e}.
\newblock Pseudoholomorphic simple {H}arnack curves.
\newblock {\em Enseign. Math.}, 61(3-4):483--498, 2015.

\bibitem[Kap91]{Ka-Gauss}
M.~M. Kapranov.
\newblock A characterization of {$A$}-discriminantal hypersurfaces in terms of
  the logarithmic {G}auss map.
\newblock {\em Math. Ann.}, 290(2):277--285, 1991.

\bibitem[Kho78]{Kho78}
A.~G. Khovanski{\u\i}.
\newblock Newton polyhedra, and the genus of complete intersections.
\newblock {\em Funktsional. Anal. i Prilozhen.}, 12(1):51--61, 1978.

\bibitem[KS15]{Kummer}
M.~Kummer and E.~Shamovich.
\newblock Real fibered morphisms and ulrich sheaves.
\newblock arXiv:1507.06760, 2015.

\bibitem[Lan15]{Lang15}
L.~Lang.
\newblock A generalization of simple {H}arnack curves.
\newblock arXiv:1504.07256, 2015.

\bibitem[Mik00]{Mik-00}
G.~Mikhalkin.
\newblock Real algebraic curves, the moment map and amoebas.
\newblock {\em Ann. of Math. (2)}, 151(1):309--326, 2000.

\bibitem[Mik01]{Mik-amoebas}
G.~Mikhalkin.
\newblock Amoebas of algebraic varieties.
\newblock {\em arXiv:math/0108225}, 2001.

\bibitem[PR11]{PasRis11}
M.~Passare and J-J. Risler.
\newblock On the curvature of the real amoeba.
\newblock In {\em Proceedings of the {G}\"okova {G}eometry-{T}opology
  {C}onference 2010}, pages 129--134. Int. Press, Somerville, MA, 2011.

\end{thebibliography}

\end{document}